\documentclass[12pt]{amsart}
\usepackage{amssymb,graphics}
\usepackage{hyperref}
\usepackage{tabularx}
\usepackage{booktabs}
\usepackage{caption}
\usepackage{mathdots}
\usepackage[usenames,dvipsnames]{xcolor}

\newtheorem{thm}{Theorem}[section]
\newtheorem{lem}[thm]{Lemma}

\newtheorem{prop}[thm]{Proposition}
\newtheorem{ex}[thm]{Example}

\newtheorem*{prob*}{Open problem}

\theoremstyle{definition}

\newtheorem{defi}[thm]{Definition}

\theoremstyle{remark}

\newtheorem*{rem*}{Remark}

\DeclareMathOperator{\id}{id}

\DeclareMathOperator{\rad}{rad}

\DeclareMathOperator{\Aut}{Aut}

\newcommand{\kringel}{\mathbin{\raise1pt\hbox{$\scriptstyle\circ$}}}
\newcommand{\pkt}{\mathbin{\raise0pt\hbox{$\scriptstyle\bullet$}}}

\newcommand{\C}{\mathbb{C}}

\newcommand{\ad}{{\rm ad}}

\newcommand{\End}{{\rm End}}
\newcommand{\Der}{{\rm Der}}

\newcommand{\nil}{\mathop{\rm nil}}

\newcommand{\La}{\mathfrak{a}}
\newcommand{\Lb}{\mathfrak{b}}

\newcommand{\Lf}{\mathfrak{f}}
\newcommand{\Lg}{\mathfrak{g}}
\newcommand{\Lh}{\mathfrak{h}}

\newcommand{\Ll}{\mathfrak{l}}
\newcommand{\Ln}{\mathfrak{n}}
\newcommand{\Lm}{\mathfrak{m}}

\newcommand{\Lr}{\mathfrak{r}}
\newcommand{\Ls}{\mathfrak{s}}

\newcommand{\CA}{\mathcal{A}}

\newcommand{\im}{\mathop{\rm im}}

\newcommand{\ep}{\varepsilon}

\newcommand{\la}{\lambda}
\newcommand{\om}{\omega}

\newcommand{\ra}{\rightarrow}
\newcommand{\ck}{\checkmark}
\renewcommand{\phi}{\varphi}

\begin{document}


\title[Perfect Lie algebras]{Post-Lie algebra structures for perfect Lie algebras}
\author[D. Burde]{Dietrich Burde}
\author[K. Dekimpe]{Karel Dekimpe}
\author[M. Monadjem]{Mina Monadjem}
\address{Fakult\"at f\"ur Mathematik\\
Universit\"at Wien\\
Oskar-Morgenstern-Platz 1\\
1090 Wien \\
Austria}
\email{dietrich.burde@univie.ac.at}
\email{mina.monadjem@univie.ac.at}
\address{Katholieke Universiteit Leuven Kulak\\
E. Sabbelaan 53 bus 7657\\
8500 Kortrijk\\
Belgium}
\email{karel.dekimpe@kuleuven.be}
\date{\today}

\subjclass[2000]{Primary 17B30, 17D25}
\keywords{Post-Lie algebra structure, perfect Lie algebra}

\begin{abstract}
We study the existence of post-Lie algebra structures on pairs of Lie algebras $(\Lg,\Ln)$, where 
one of the algebras is perfect non-semisimple, and the other one is abelian, nilpotent non-abelian, 
solvable non-nilpotent, simple, semisimple non-simple, reductive non-semisimple or complete non-perfect. 
We prove several non-existence results, but also provide examples in some cases 
for the existence of a post-Lie algebra structure. 
Among other results we show that there is no post-Lie algebra structure on $(\Lg,\Ln)$, where 
$\Lg$ is perfect non-semisimple, and $\Ln$ is $\mathfrak{sl}_3(\C)$. We also show that there 
is no post-Lie algebra structure on $(\Lg,\Ln)$, where $\Lg$ is perfect and $\Ln$ is reductive 
with a $1$-dimensional center. 
\end{abstract}

\maketitle

\section{Introduction}

Post-Lie algebras and Post-Lie algebra structures (or PA-structures) on pairs of Lie algebras have 
been studied in many areas of mathematics during the last years. PA-structures are a natural 
generalization of pre-Lie algebra structures on Lie algebras, which arise among other things from 
affine manifolds and affine actions on Lie groups, crystallographic groups, \'etale affine 
representations of Lie algebras, quantum field theory, operad theory, Rota-Baxter operators, 
and deformation theory of rings and algebras. There is a large literature on pre-Lie and post-Lie 
algebras, see for example \cite{BU41,BU44,BU51,BU58,BU59,BU73,ELM,VAL} and the references therein. 
For a survey on pre-Lie algebra respectively post-Lie algebra structures see \cite{BU24, BU65}.\\[0.2cm]
In the present article we study the existence question of post-Lie algebra structures
on pairs of Lie algebras $(\Lg,\Ln)$, where one Lie algebra is perfect and the other one is
abelian, nilpotent, solvable, simple, semisimple, reductive, complete or perfect.
In all but three cases we can solve the existence question and generalize our 
previous results for semisimple Lie algebras to perfect Lie algebras. 
However, for these three cases of $(\Lg,\Ln)$, namely where $\Lg$ is perfect non-semisimple,
and $\Ln$ is either nilpotent, simple, or semisimple, we are not able solve the existence 
question in general. We conjecture that there do not exist post-Lie algebra structures 
in these cases. For some special families of examples we can prove this conjecture. \\[0.2cm]
The outline of this paper is as follows. In the second section we provide basic results on
perfect Lie algebras, including a classification of complex perfect Lie algebras of 
dimension $n\le 9$. We also recall the basic notions for post-Lie algebra structures. 
In the third section we study the existence question for pairs $(\Lg,\Ln)$ where $\Lg$ is perfect.
If $\Ln$ is nilpotent, then we show for perfect Lie algebras $\Lg$ of dimension $6$, that there exist
no post-Lie algebra structures on $(\Lg,\Ln)$. We also prove, using the classification of complex perfect
Lie algebras in low dimension, that there is no post-Lie algebra structure on $(\Lg,\Ln)$, where 
$\Ln=\mathfrak{sl}_3(\C)$. We find post-Lie algebra structures on $(\Lg,\Ln)$ for examples
of reductive Lie algebras $\Ln$, and show that such structures do not exist for reductive
Lie algebras $\Ln$ with a $1$-dimensional center. In the fourth section we study the existence 
question for pairs $(\Lg,\Ln)$ where $\Ln$ is perfect. Here we often find post-Lie algebra 
structures and are able to solve the existence problem in all cases.

\section{Preliminaries}

Let $\Lg$ be a finite-dimensional Lie algebra over a field $K$. Denote by $Z(\Lg)$ the center of $\Lg$,
by $\rad(\Lg)$ the solvable radical of $\Lg$, and by $\nil(\Lg)$ the nilradical of $\Lg$. A Lie algebra $\Lg$ is
called {\em perfect}, if $\Lg=[\Lg,\Lg]$. Every semisimple Lie algebra over a field of characteristic zero is perfect.
The converse does not hold. It is well known that the solvable radical of a perfect Lie algebra is nilpotent. One can also
give a necessary and sufficient condition for a Levi decomposition of $\Lg$, so that $\Lg$ is perfect.

\begin{prop}\label{2.1}
Let $\Lg=\Ls\ltimes \rad(\Lg)$ be a Levi decomposition of a Lie algebra $\Lg$ and consider 
$V=\Lr/[\Lr,\Lr]$ with $\Lr=\rad(\Lg)$ as an $\Ls$-module. Then $\Lg$ is perfect if and 
only if $V$ does not contain the trivial $1$-dimensional $\Ls$-module.
\end{prop}  

The lowest-dimensional example of a complex perfect non-semisimple Lie algebra is
$\mathfrak{sl}_2(\C)\ltimes V(2)$ of dimension $5$.
Here $\mathfrak{sl}_2(\C)={\rm span}\{ e_1,e_2,e_3\}$ has Lie brackets given by
$[e_1,e_2]=e_3$, $[e_1,e_3] = -2e_1$, $[e_2,e_3] = 2e_2$, and $V(2)={\rm span}\{e_4,e_5\}$ 
is the natural representation of $\mathfrak{sl}_2(\C)$. The Lie brackets of 
$\mathfrak{sl}_2(\C)\ltimes V(2)$ are given by
\vspace*{0.5cm}
\begin{align*}
[e_1,e_2] & = e_3,     & [e_2,e_3] & = 2e_2,    & [e_3,e_5]& =-e_5.\\
[e_1,e_3] & = -2e_1,   & [e_2,e_4] & = e_5,     &                  \\
[e_1,e_5] & = e_4,     & [e_3,e_4] & = e_4,     &                   \\
\end{align*}
This Lie algebra is perfect and has a non-trivial solvable radical.
Its center is trivial. We also give an example of a non-semisimple perfect
Lie algebra with non-trivial center.
For this, let $\mathfrak{sl}_2(\C)={\rm span}\{ e_1,e_2,e_3\}$
and $\mathfrak{n}_3(\C)={\rm span}\{ e_4,e_5,e_6\}$ be the Heisenberg Lie algebra, with
$[e_4,e_5]=e_6$. Consider the following semidirect sum of $\mathfrak{sl}_2(\C)$ and 
$\Ln_3(\C)$, given by the following Lie brackets in the basis  $(e_1,\ldots ,e_6)$:
\vspace*{0.5cm}
\begin{align*}
[e_1,e_2] & = e_3,     & [e_2,e_3] & = 2e_2,    & [e_3,e_5]& =-e_5,\\
[e_1,e_3] & = -2e_1,   & [e_2,e_4] & = e_5,     & [e_4,e_5] & = e_6.\\
[e_1,e_5] & = e_4,     & [e_3,e_4] & = e_4,     &                   \\
\end{align*}
We denote this Lie algebra by $\mathfrak{sl}_2(\C) \ltimes \Ln_3(\C)$.

\begin{ex}\label{2.2}
The Lie algebra $\Lg=\mathfrak{sl}_2(\C) \ltimes \Ln_3(\C)$ is perfect, but not semisimple. It has a $1$-dimensional
center.
\end{ex}  

Indeed, we have $Z(\Lg)={\rm span}\{ e_6\}$, and $\Lg$ is not semisimple. The nilradical $\Ln_3(\C)$ is isomorphic to
$V(2)\oplus V(1)$ as $\mathfrak{sl}_2(\C)$-module, where $V(n)$ denotes the irreducible $\mathfrak{sl}_2(\C)$-module of
dimension $n\ge 1$. It follows from the Lie brackets that $\Lg$ is perfect. We also can derive this from
Proposition $\ref{2.1}$. The quotient $V=\Ln_3/[\Ln_3,\Ln_3]$ is isomorphic to $V(2)$
and does not contain the trivial $1$-dimensional module $V(1)$. Hence $\Lg$ is perfect. \\[0.2cm]
In the study of PA-structures for perfect Lie algebras we are also interested in a classification of perfect Lie algebras
in low dimensions. Turkowski has classified Lie algebras with non-trivial Levi decomposition up to dimension $8$ over the real
numbers in \cite{TUR1}, where he lists explicit Lie brackets for all algebras. From this work it is not difficult to derive a
classification of complex perfect Lie algebras of dimension $n\le 8$. We need to add one Lie algebra though, which Turkowski
has not in his list. It is the complexification of the algebra $L_{8,13}^{\ep}$ for $\ep=0$, isomorphic to 
$\Ls\Ll_2(\C)\ltimes (V(2)\oplus \Ln_3(\C))$. Turkowski only allows $\ep=\pm 1$. \\[0.2cm]
There is another classification given by Alev, Ooms and Van den Bergh in \cite{ALE}, namely the classification of 
non-solvable algebraic Lie algebras of dimension $n\le 8$ over an algebraically closed field of 
characteristic zero. It also contains explicit Lie brackets for all algebras. Since perfect
Lie algebras are algebraic, we obtain again a list of complex perfect non-semisimple Lie algebras
of dimension $n\le 8$. This list coincides with the (corrected) one by Turkowski. Let $\Ln_5(\C)$ denote the
$5$-dimensional Heisenberg Lie algebra, with basis $(e_1,\ldots ,e_5)$ and Lie brackets
\[
[e_1,e_2]=[e_3,e_4]=e_5.
\]  
Denote by $\Lf_{2,3}(\C)$ the free-nilpotent Lie algebra with $2$ generators and nilpotency class $3$, with basis
$(e_1,\ldots, e_5)$ and Lie brackets
\[
[e_1,e_2]=e_3,\; [e_1,e_3]=e_4,\; [e_2,e_3]=e_5.
\]  
The classification result is as follows.

\begin{prop}\label{2.3}
Every complex perfect non-semisimple Lie algebra of dimension $n\le 8$ is isomorphic to one of the 
following Lie algebras:
\vspace*{0.5cm}
\begin{center}
\begin{tabular}{c|cccc}
\color{red}{$\Lg$} & \color{blue}{$\dim \Lg$} & \color{purple}{$\dim Z(\Lg)$} 
& \color{ForestGreen}{Turkowski} & \color{orange}{Alev et al.} \\[4pt]
\hline
$\Ls\Ll_2(\C)\ltimes V(2)$ & $5$ & $0$ & $L_{5,1}$  & $L_5$ \\[4pt]
$\Ls\Ll_2(\C)\ltimes V(3)$ & $6$ & $0$ & $L_{6,4}\cong L_{6,1}$ & $L_{6,1}$\\[4pt]
$\Ls\Ll_2(\C)\ltimes \Ln_3(\C)$  & $6$ & $1$ & $L_{6,2}$  & $L_{6,3}$  \\[4pt]
$\Ls\Ll_2(\C)\ltimes V(4)$ & $7$ & $0$ & $L_{7,6}$  & $L_{7,1}$  \\[4pt]
$\Ls\Ll_2(\C)\ltimes (V(2)\oplus V(2))$ & $7$ & $0$ & $L_{7,7}$  & $L_{7,2}$ \\[4pt]
$\Ls\Ll_2(\C)\oplus (\Ls\Ll_2(\C)\ltimes V(2))$ & $8$ & $0$ & $\Ls\Ll_2\oplus L_{5,1}$ 
& $\mathfrak{sl}_2\oplus L_5$ \\[4pt]
$\Ls\Ll_2(\C)\ltimes V(5)$ & $8$ & $0$ & $L_{8,21}$  & $L_{8,1}$  \\[4pt]
$\Ls\Ll_2(\C)\ltimes (V(2)\oplus V(3))$ & $8$ & $0$ & $L_{8,22}$  & $L_{8,2}$ \\[4pt]
$\Ls\Ll_2(\C)\ltimes (V(2)\oplus \Ln_3(\C))$ & $8$ & $1$ & $\color{red}{L_{8,13}^{\ep=0}}$  
& $L_{8,13}$  \\[4pt]
$\Ls\Ll_2(\C)\ltimes \Lf_{2,3}(\C)$ & $8$ & $0$ & $L_{8,15}$  & $L_{8,18}$ \\[4pt]
$\Ls\Ll_2(\C)\ltimes_{\phi} \Ln_5(\C)$ & $8$ & $1$ & $L_{8,13}^1\cong L_{8,13}^{-1}$  & $L_{8,15}$ \\[4pt]
$\Ls\Ll_2(\C)\ltimes_{\psi} \Ln_5(\C)$ & $8$ & $1$ & $L_{8,19}$  & $L_{8,16}$ \\[4pt]
\end{tabular}
\end{center}
\vspace*{0.5cm}
\end{prop}  

Turkowski has also classified in \cite{TUR2} real and complex Lie algebras with 
non-trivial Levi decomposition in dimension $9$. This yields a list of complex perfect 
Lie algebras of dimension $9$.
Denote by $\Lf_{3,2}(\C)$ the free-nilpotent Lie algebra with $3$ generators and nilpotency class $2$, and
by $\CA_{6,4}(\C)$ the $2$-step nilpotent Lie algebra with basis $(e_1,\ldots ,e_6)$ and Lie brackets
\[
[e_1,e_4]=e_5,\; [e_2,e_3]=-e_5,\; [e_3,e_4]=e_6.
\]
Then we have the following result.

\begin{prop}\label{2.4}
Every complex perfect non-semisimple Lie algebra of dimension $9$ is isomorphic to one of the following Lie algebras:

\vspace*{0.5cm}
\begin{center}
\begin{tabular}{c|ccc}
\color{red}{$\Lg$} & \color{blue}{$\dim \Lg$} & \color{purple}{$\dim Z(\Lg)$} & \color{ForestGreen}{Turkowski}  \\[4pt]
\hline
$\Ls\Ll_2(\C)\oplus (\Ls\Ll_2(\C)\ltimes V(3))$ & $9$ & $0$ & $\Ls\Ll_2\oplus L_{6,1}$  \\[4pt]
$\Ls\Ll_2(\C)\oplus (\Ls\Ll_2(\C)\ltimes \Ln_3(\C))$  & $9$ & $1$ & $\Ls\Ll_2\oplus L_{6,2}$ \\[4pt]
$\Ls\Ll_2(\C)\ltimes V(6)$ & $9$ & $0$ & $L_{9,59}$  \\[4pt]
$\Ls\Ll_2(\C)\ltimes (V(2)\oplus V(4))$ & $9$ & $0$ & $L_{9,60}$  \\[4pt]
$\Ls\Ll_2(\C)\ltimes (V(3)\oplus V(3))$ & $9$ & $0$ & $L_{9,61}$  \\[4pt]
$\Ls\Ll_2(\C)\ltimes (V(2)\oplus V(2)\oplus V(2))$ & $9$ & $0$ & $L_{9,63}$  \\[4pt]
$\Ls\Ll_2(\C)\ltimes (V(3)\oplus \Ln_3(\C))$ & $9$ & $1$ & $L_{9,58}$  \\[4pt]
$\Ls\Ll_2(\C)\ltimes (\Ln_3(\C)\oplus \Ln_3(\C))$ & $9$ & $2$ & $L_{9,37}\cong L_{9,42}$  \\[4pt]
$\Ls\Ll_2(\C)\ltimes \Lf_{3,2}(\C)$ & $9$ & $0$ & $L_{9,62}$  \\[4pt]
$\Ls\Ll_2(\C)\ltimes \CA_{6,4}(\C)$ & $9$ & $2$ & $L_{9,41}$  \\[4pt]
\end{tabular}
\end{center}
\vspace*{0.5cm}
In particular the nilradical of such an algebra always has nilpotency class $c\le 2$.
\end{prop}

We recall the definition of a post-Lie algebra structure on a pair of Lie algebras $(\Lg,\Ln)$ over a field
$K$, see \cite{BU41}:

\begin{defi}\label{pls}
Let $\Lg=(V, [\, ,])$ and $\Ln=(V, \{\, ,\})$ be two Lie brackets on a vector space $V$ over
$K$. A {\it post-Lie algebra structure}, or {\em PA-structure} on the pair $(\Lg,\Ln)$ is a
$K$-bilinear product $x\cdot y$ satisfying the identities:
\begin{align}
x\cdot y -y\cdot x & = [x,y]-\{x,y\} \label{post1}\\
[x,y]\cdot z & = x\cdot (y\cdot z) -y\cdot (x\cdot z) \label{post2}\\
x\cdot \{y,z\} & = \{x\cdot y,z\}+\{y,x\cdot z\} \label{post3}
\end{align}
for all $x,y,z \in V$.
\end{defi}

Define by  $L(x)(y)=x\cdot y$ the left multiplication operators of the algebra $A=(V,\cdot)$.
By \eqref{post3}, all $L(x)$ are derivations of the Lie algebra $(V,\{,\})$. Moreover, by \eqref{post2}, the left multiplication
\[
L\colon \Lg\ra \Der(\Ln)\subseteq \End (V),\; x\mapsto L(x)
\]
is a linear representation of $\Lg$. \\
If $\Ln$ is abelian, then a post-Lie algebra structure on $(\Lg,\Ln)$ corresponds to a {\it pre-Lie algebra structure},
or left-symmetric structure on $\Lg$. In other words, if $\{x,y\}=0$ for all $x,y\in V$, then the conditions reduce to
\begin{align}
x\cdot y-y\cdot x & = [x,y] \label{pre1} \\
[x,y]\cdot z & = x\cdot (y\cdot z)-y\cdot (x\cdot z), \label{pre2}
\end{align}
i.e., $x\cdot y$ is a {\it pre-Lie algebra structure} on the Lie algebra $\Lg$. \\[0.2cm]
For semisimple Lie algebras $\Ln$ we have the following result on PA-structures on pairs $(\Lg,\Ln)$, see
Proposition $2.14$ in \cite{BU41}.

\begin{prop}\label{2.6}
Let $\Ln$ be a semisimple Lie algebra. Then a pair $(\Lg,\Ln)$ admits a PA-structure if and only if there is
an injective Lie algebra homomorphism $\phi\colon \Lg \hookrightarrow \Ln\oplus \Ln$ such that the map
$(p_1-p_2)_{\mid \Lg} \colon \Lg\ra \Ln$ is bijective. Here $p_i\colon \Ln\oplus \Ln \ra \Ln$ denotes the projection
onto the $i$-th factor for $i=1,2$.
\end{prop}  

Let us denote the composition of $p_i$ and $\phi$ by $j_i= \phi\circ p_i\colon \Lg\ra \Ln$  for $i=1,2$. \\[0.2cm]
We also recall the following results, see Propositions $2.14$ and $2.21$ in \cite{BU59}.

\begin{prop}\label{2.7}
Let $(\Lg,\Ln)$ be a pair of Lie algebras, where $\Ln$ is complete. Then there is a bijection between PA-structures
on $(\Lg,\Ln)$ and Rota-Baxter operators $R$ of weight $1$ on $\Ln$. Every such PA-structure is then of the form
$x\cdot y=\{R(x),y\}$. If $\Lg$ and $\Ln$ are not isomorphic, then both $\ker(R)$ and $\ker(R+\id)$ are nonzero ideals in $\Lg$.
\end{prop}  

Here a Rota-Baxter algebra operator, for a nonassociative algebra over a field $K$, of weight $\la\in K$ is a linear
operator $R\colon A\ra A$ satisfying
\[
R(x)R(y)=R(R(x)y+xR(y)+\la xy) 
\]
for all $x,y\in A$.

\section{PA-structures with $\Lg$ perfect}

In this section we study the existence question of PA-structures on pairs of complex Lie algebras $(\Lg,\Ln)$, where
$\Lg$ is perfect non-semisimple. We consider $7$ different cases for $\Ln$, namely $(a)$ $\Ln$ is abelian,
$(b)$ $\Ln$ is nilpotent non-abelian, $(c)$ $\Ln$ is solvable non-nilpotent, $(d)$ $\Ln$ is simple,
$(e)$ $\Ln$ is semisimple non-simple, $(f)$ $\Ln$ is reductive non-semisimple, and $(g)$ $\Ln$ is complete non-perfect. \\[0.2cm]
We start with case $(a)$.

\begin{prop}\label{3.1}
There is no PA-structure on a pair $(\Lg,\Ln)$, where $\Lg$ is perfect and $\Ln$ is abelian.
\end{prop}

\begin{proof}
A PA-structure on a pair $(\Lg,\Ln)$, where $\Ln$ is abelian, corresponds to a left-symmetric (or pre-Lie algebra)
structure on $\Lg$. By Corollary $21$ of \cite{HEL} there is no such structure on a perfect Lie algebra.
\end{proof}  

For case $(b)$ we only have partial results so far. In Proposition $3.6$ of \cite{BU51} we have proved
that there is no post-Lie algebra structure on $(\Lg,\Ln)$,  where $\Lg$ is perfect of dimension $5$, namely
$\Lg=\Ls\Ll_2(\C)\ltimes V(2)$, and $\Ln$ is nilpotent. We can generalize this result to perfect Lie algebras
$\Lg$ of dimension $6$. According to Proposition $2.3$, the non-semisimple perfect Lie algebras of dimension $6$ are given
by $\mathfrak{sl}_2(\C)\ltimes V(3)$ and $\mathfrak{sl}_2(\C)\ltimes \Ln_3(\C)$.

\begin{prop}\label{3.2}
Let $(\Lg,\Ln)$ be a pair of Lie algebras, where $\Lg$ is either $\Ls\Ll_2(\C)\ltimes V(3)$ or $\Ls\Ll_2(\C)\ltimes \Ln_3(\C)$,
and $\Ln$ is nilpotent. Then there is no PA-structure on $(\Lg,\Ln)$.
\end{prop}

\begin{proof}
Let $\Lg=\mathfrak{sl}_2(\C)\ltimes \La$ and sssume that there exists a PA-structure on $(\Lg,\Ln)$, with the
homomorphism $L\colon \Lg\ra \Der(\Ln)$ given as in Definition $\ref{pls}$.
Let $\phi \colon \Lg \hookrightarrow \Ln \rtimes \Der(\Ln)$ be the embedding defined by $x\mapsto (x,L(x))$.
We claim that $\ker(L) \cap \mathfrak{sl}_2(\C)=0$. Suppose that this intersection is nonzero. Since it is an ideal
in the simple Lie algebra $\mathfrak{sl}_2(\C)$, this implies that $\ker(L) \cap \mathfrak{sl}_2(\C)= \mathfrak{sl}_2(\C)$,
so that  $\mathfrak{sl}_2(\C)\subseteq \ker(L)$. In particular, we have $L(s)=0$ for all $s\in \mathfrak{sl}_n(\C)$.
By axiom $(1)$ in Definition $\ref{pls}$ it follows that
\[
[s,t]-\{s,t\}=s\cdot t-t\cdot s=0
\]
for all $s,t\in \mathfrak{sl}_2(\C)$. Hence $\Ln$ has a subalgebra isomorphic to $\mathfrak{sl}_2(\C)$. This is
impossible, because $\Ln$ is nilpotent, so that the claim follows.

We have shown in the proof of Theorem $3.3$ in \cite{BU73} that the  Lie algebra
$\Ln \rtimes \Lh$, with 
\[
\Lh=L(\Lg)\subseteq \Der(\Ln),
\]
has a direct vector space sum decomposition
\[ 
\Ln \rtimes \Lh =\phi(\Lg)\dotplus \Lh.
\] 
Since $\phi(\Lg)$ and $\Lh$ are homomorphic images of a perfect Lie algebra, they are perfect.
Hence $\Ln \rtimes \Lh$ is perfect. Let $\Ls=L(\mathfrak{sl}_2(\C))$ and $\Lr=L(\La)$. Then we have $\Lh=\Ls\ltimes \Lr$.
Since
$\ker(L)\cap \mathfrak{sl}_2(\C)= 0$, $\Ls$ is nonzero and semisimple. Hence $\dim (\Ls)=3$ and $\Ls\cong \mathfrak{sl}_2(\C)$.
So
\[
\Lh\ltimes \Ln=(\Ls\ltimes \Lr)\ltimes \Ln
\]
is perfect and has nilradical $\Lm=\Lr\dotplus \Ln$. By Proposition $\ref{2.1}$, $\Lm/[\Lm,\Lm]$ does not contain
the trivial $1$-dimensional $\Ls$-module $V(1)$. Since $\Ln\subseteq \Lm$, also $\Ln/[\Ln,\Ln]\subseteq \Lm/[\Lm,\Lm]$
does not contain the trivial $1$-dimensional $\Ls$-module $V(1)$. Hence $\Ls\ltimes \Ln$ is a perfect Lie algebra of dimension $9$.
By Proposition $\ref{2.4}$, the nilpotency class $c(\Ln)$ of $\Ln$ is at most $2$. Proposition $4.2$ of \cite{BU41} says, that
if $(\Lg,\Ln)$ admits a post-Lie algebra structure, and $c(\Ln)\le 2$, then $\Lg$ admits a pre-Lie algebra structure.
Since $\Lg$ is perfect, this is impossible by Corollary $21$ of \cite{HEL}.
\end{proof}

Let us again state the last result used in the proof, see also Proposition $3.3$ in \cite{BU51}. 

\begin{prop}
Let $(\Lg,\Ln)$ be a pair of Lie algebras, where $\Lg$ is perfect and $\Ln$ is $2$-step nilpotent. Then there exists
no PA-structure on $(\Lg,\Ln)$.  
\end{prop}

In case $(c)$, we have proved the following result in Proposition $4.4$ of \cite{BU41}. 

\begin{prop}
There is no PA-structure on a pair $(\Lg,\Ln)$, where $\Lg$ is perfect and $\Ln$ is solvable non-nilpotent.
\end{prop}
 
For case $(d)$ we start with low-dimensional simple Lie algebras $\Ln$. There is no pair
$(\Lg,\Ln)$ with $\Ln\cong  \mathfrak{sl}_2(\C)$ and $\Lg$ perfect non-semisimple, since the only perfect Lie
algebra in dimension $3$ is simple. The next case is to consider pairs $(\Lg,\Ln)$, where
$\Ln\cong \mathfrak{sl}_3(\C)$ and $\Lg$ is a perfect non-semisimple Lie algebra of dimension $8$.
We start with the following result.

\begin{lem}\label{3.4}
Let $i\colon \mathfrak{sl}_2(\C) \ltimes V(2) \hookrightarrow \mathfrak{sl}_3(\C)$ be an injective Lie algebra homomorphism.
By conjugating with a matrix in $GL_3(\C)$ we may assume that the image of $i$ is of the form
\[
\im (i)=\Biggl\{ \begin{pmatrix} a_1 & a_2 & a_3 \cr  a_4 & -a_1 & a_5 \cr  0 & 0 & 0 \end{pmatrix} \in \mathfrak{sl}_3(\C)
\Biggr\}, \text{ or }
\im (i)=\Biggl\{ \begin{pmatrix} b_1 & b_2 & 0 \cr  b_3 & -b_1 & 0 \cr  b_4 & b_5 & 0 \end{pmatrix} \in \mathfrak{sl}_3(\C)
\Biggr\} .
\]
\end{lem}  

\begin{proof}
The vector space $\C^3$ becomes an $\mathfrak{sl}_2(\C)$-module via the embedding $i$ restricted to $\mathfrak{sl}_2(\C)$.
Hence either $\C^3\cong V(2)\oplus V(1)$, or  $\C^3\cong V(3)$ as $\mathfrak{sl}_2(\C)$-module. In the first case we can
choose a basis of $\C^3$ such that
\[
i(\mathfrak{sl}_2(\C)) = \Biggl\{ \begin{pmatrix} a_1 & a_2 & 0 \cr  a_3 & -a_1 & 0 \cr  0 & 0 & 0 \end{pmatrix}
\in \mathfrak{sl}_3(\C)\Biggr\},
\]  
by using the natural representation $i(e_1)=E_{12}$, $i(e_2)=E_{21}$ and $i(e_3)=E_{11}-E_{22}$ for the basis $(e_1,e_2,e_3)$
of $\mathfrak{sl}_2(\C)$. A short computation shows that when such a representation extends to
$\mathfrak{sl}_2(\C) \ltimes V(2)$, we obtain one of the forms for $\im (i)$ as described above. \\[0.2cm]
In the second case we may assume that $i(e_1)=E_{12}+2E_{23}$, $i(e_2)=E_{12}+2E_{23}$ and $i(e_3)=2E_{11}-2E_{33}$.
It is easy to see that this representation does not extend to one of $\mathfrak{sl}_2(\C) \ltimes V(2)$.
\end{proof}  

\begin{lem}\label{3.5}
Let $j\colon \mathfrak{sl}_2(\C) \hookrightarrow \mathfrak{sl}_3(\C)$ be an injective Lie algebra homomorphism.
Denote by $c_i\colon \mathfrak{gl}_3(\C)\ra \C^3$ the projection of a matrix in $\mathfrak{gl}_3(\C)$ to its $i$-th
column, and by $r_i\colon \mathfrak{gl}_3(\C)\ra \C^3$ the projection to its $i$-th row. Then none of the linear maps
$c_i \circ j, r_i \circ j \colon \mathfrak{sl}_2(\C)\ra \C^3$ for $i=1,2,3$ is bijective.
\end{lem}  

\begin{proof}
It is enough to show the claim for columns. We obtain the result for rows by applying the isomorphism of
Lie algebras $\mathfrak{sl}_3(\C)\ra \mathfrak{sl}_3(\C)$, given by $X\mapsto -X^T$, to the result for columns.
We will give the proof for $c_3\circ j$. The other two cases are similar. Note that $\C^3$ is an
$\mathfrak{sl}_2(\C)$-module via $j$. Hence the map $c_i \circ j$ is actually the map
\[
\mathfrak{sl}_2(\C) \ra \C^3,\;x\mapsto x\cdot \begin{pmatrix} 0 \cr 0 \cr 1 \end{pmatrix}.
\]
Because of $\dim \mathfrak{sl}_2(\C) = \dim \C^3$ the annihilator of any vector is non-trivial by Lemma $4.1$ in \cite{BU44},
so that the map $c_3\cdot j$ is not injective.
\end{proof}  

\begin{lem}\label{3.6}
Let $\Ls_1$ and $\Ls_2$ be Lie algebras isomorphic to $\mathfrak{sl}_2(\C)$. Then the ideals
of the Lie algebra  $\Lg=\Ls_1\oplus (\Ls_2\ltimes V(2))$ are given by
\[
0,\, \Ls_1,\, \Ls_2\ltimes V(2),\,  \Ls_1\oplus V(2),\,  V(2),\,  \Lg.
\]  
\end{lem}  

\begin{proof}
It is clear that all of these subspaces are ideals in $\Lg$. Conversely, assume that $\La$ is an ideal in $\Lg$.
If $\La\cap \Ls_1\neq 0$, then $\Ls_1\subseteq \La$ and $\La/\Ls_1$ is an ideal of
$\Lg/\Ls_1\cong \Ls_2\ltimes V(2)$. But the only ideals of $\Ls_2\ltimes V(2)$ are $0$, $V(2)$ and $\Ls_2\ltimes V(2)$.
So, if $\La\cap \Ls_1\neq 0$ then all ideals are given by $\Ls_1$, $\Ls_1\oplus V(2)$ and $\Lg$. \\[0.2cm]
Now suppose that $\La\cap \Ls_1= 0$. We claim that then $\La\subseteq \Ls_2\ltimes V(2)$. Indeed, suppose that
there exists an element $x$ in $\La\setminus ( \Ls_2\ltimes V(2))$. We can write $x=x_1+x_2$ with $x_1\in \Ls_1$
and $x_2\in \Ls_2\ltimes V(2)$, where $x_1\neq 0$. There exists a $y\in \Ls_1$ such that $[y,s_1]\neq 0$, so that
$0\neq [y,x]=[y,x_1+x_2]=[y,x_1]\in \Ls_1\cap \La$, which is a contradiction. Hence we have
$\La\subseteq \Ls_2\ltimes V(2)$, which leads to the ideals $0$, $V(2)$ and $\Ls_2\ltimes V(2)$.
\end{proof}  

\begin{lem}\label{3.7}
There is no direct vector space decomposition  $\mathfrak{sl}_3(\C)=\La\dotplus \Lb$ with subalgebras
$\La$ and $\Lb$ of $\mathfrak{sl}_3(\C)$ satisfying $\La\cong  \mathfrak{sl}_2(\C) \ltimes V(2)$ and
$\Lb\cong  \mathfrak{sl}_2(\C)$.
\end{lem}  

\begin{proof}
Assume that there is such a decomposition $\mathfrak{sl}_3(\C)=\La\dotplus \Lb$. Then after applying a base
change we may assume by Lemma $\ref{3.4}$ that 
\[
\La=\Biggl\{ \begin{pmatrix} a_1 & a_2 & a_3 \cr  a_4 & -a_1 & a_5 \cr  0 & 0 & 0 \end{pmatrix} \in \mathfrak{sl}_3(\C)
\Biggr\}, \text{ or }
\La=\Biggl\{ \begin{pmatrix} b_1 & b_2 & 0 \cr  b_3 & -b_1 & 0 \cr  b_4 & b_5 & 0 \end{pmatrix} \in \mathfrak{sl}_3(\C)
\Biggr\}
\]
As $\mathfrak{sl}_3(\C)$ is a direct vector space sum of $\La$ and $\Lb$ we must have
that the row projection map $r_3\colon \Lb\ra \C^3$ is bijective in the first case, and the column projection map
$c_3\colon \Lb\ra \C^3$ is bijective in the second case. However, by Lemma $\ref{3.5}$, this is impossible.
\end{proof}  

We can now apply these lemmas to PA-structures on pairs $(\Lg,\Ln)$ with $\Ln =\mathfrak{sl}_3(\C)$, where $\Lg$
has a levi subalgebra isomorphic to $\mathfrak{sl}_2(\C)\oplus \mathfrak{sl}_2(\C)$.

\begin{prop}\label{3.8}
Let $\Lg=\mathfrak{sl}_2(\C)\oplus (\mathfrak{sl}_2(\C)\ltimes V(2))$ and $\Ln =\mathfrak{sl}_3(\C)$. Then there is no
PA-structure on the pair $(\Lg,\Ln)$.
\end{prop}  

\begin{proof}
Assume that there exists a  PA-structure on $(\Lg,\Ln)$ with $\Ln =\mathfrak{sl}_3(\C)$. Let us write
$\Lg=\Ls_1\oplus (\Ls_2\ltimes V(2))$. By Proposition $\ref{2.6}$ there exists an injective Lie algebra homomorphism
$j\colon \Lg\ra \Ln\oplus \Ln$, $x\mapsto (j_1(x),j_2(x))$ such that $j_1-j_2$ is a bijective linear map. We will examine
the possible kernels of $j_1$. Since $\ker(j_1)$ is an ideal in $\Lg$, it must be one of the six possibilities given
in Lemma $\ref{3.6}$. \\[0.2cm]
{\em Case 1:} $\ker(j_1)=0$. Then $j_1\colon \Lg\ra \Ln$ is an isomorphism. Since $\Lg$ is not semisimple, this
is a contradiction. \\[0.2cm]
{\em Case 2:}  $\ker(j_1)=\Lg$. Then $j_1$ is the zero map. Since $j_1-j_2$ has to be bijective, $j_2\colon \Lg\ra \Ln$
is an isomorphism. This is impossible. \\[0.2cm]
{\em Case 3:}  $\ker(j_1)=V(2)$. Then the representation 
\[
j_1\mid_{\Ls_1\oplus \Ls_2}\colon \Ls_1\oplus \Ls_2 \ra \mathfrak{sl}_3(\C)
\]
is faithful. However, the smallest dimension of a faithful representation of $\Ls_1\oplus \Ls_2$ is equal to $4$,
see \cite{BU29}, Proposition $2.5$. This is a contradiction.  \\[0.2cm]
{\em Case 4:}  $\ker(j_1)=\Ls_1\oplus V(2)$. Then
\[
j_2 \mid_{\Ls_1\oplus V(2)} = (j_2-j_1)\mid_{\Ls_1\oplus V(2)} \colon \Ls_1\oplus V(2)\ra \mathfrak{sl}_3(\C)
\]
has to be injective. However, $\Ls_1\oplus V(2)$ contains a $3$-dimensional abelian subalgebra, whereas $\mathfrak{sl}_3(\C)$
does not. This is a contradiction. \\[0.2cm]
{\em Case 5:}  $\ker(j_1)=\Ls_2\ltimes V(2)$. Then $(\Ls_2\ltimes V(2)) \cap \ker(j_2)=0$. Since $\ker(j_2)$ is an ideal, which is
nonzero as in case $1$, we have $\ker(j_2)=\Ls_1$. So for every $x\in \Lg$ we can write $x=x_1+x_2$ with $x_1\in \Ls_1$ and
$x_2\in \Ls_2\ltimes V(2)$. Then we have $(j_1-j_2)(x)=j_1(x_1)-j_2(x_2)$, and $j_1-j_2$ is injective if and only if
$\im(j_1)\cap \im(j_2)=0$. This is equivalent to
\[
\mathfrak{sl}_3(\C)=\im (j_1)\dotplus \im(j_2) 
\]
with $\im(j_1)\cong \Ls_1$ and $\im(j_2)\cong \Ls_2\ltimes V(2)$, which is a contradiction to Lemma $\ref{3.7}$. \\[0.2cm]
{\em Case 6:}  $\ker(j_1)=\Ls_1$. Then $\Ls_1$ is not contained in $\ker(j_2)$. So from the six possibilities for the ideal
$\ker(j_2)$, there are left $0$, $V(2)$ and $\Ls_2\ltimes V(2)$. We already know that $\ker(j_2)$ must be nonzero. Also,
$\ker(j_2)=V(2)$ leads to a contradiction as in case $3$. Finally $\ker(j_2)=\Ls\ltimes V(2)$ and $\ker(j_1)=\Ls_1$ is exactly
the symmetric situation to case $5$, and so also leads to a contradiction.
\end{proof}

Now we can prove the following result.

\begin{thm}\label{3.9}
Let $(\Lg,\Ln)$ be a pair of Lie algebras, where $\Lg$ is perfect non-semisimple and $\Ln =\mathfrak{sl}_3(\C)$. Then there is no
PA-structure on $(\Lg,\Ln)$.
\end{thm}  

\begin{proof}
Denote by $\Ls$ a Levi subalgebra of $\Lg$. Then either $\Ls\cong \mathfrak{sl}_2(\C)\oplus \mathfrak{sl}_2(\C)$ or
$\Ls\cong \mathfrak{sl}_2(\C)$. In the first case we have $\Lg \cong \Ls\Ll_2(\C)\oplus (\Ls\Ll_2(\C)\ltimes V(2))$ by
Proposition $\ref{2.3}$. Then the claim follows by Proposition $\ref{3.8}$. In the second case, by Proposition  $\ref{2.3}$,
$\Lg \cong \mathfrak{sl}_2(\C)\ltimes \Lr$, where $\Lr$ is isomorphic to one of the following five Lie algebras
\[
V(5),\, V(2)\oplus V(3),\, V(2)\oplus \Ln_3(\C),\, \Ln_5(\C),\, \Lf_{2,3}(\C).
\]
Again we are using the maps $j_i\colon \Lg\ra \Ln$, and assume that $j_1-j_2$ is bijective. 
In the first two cases, either $j_1$ or $j_2$ must be injective on the factor $V(3)$, respectively $V(5)$. This contradicts
the fact that $\mathfrak{sl}_3(\C)$ does not contain an abelian subalgebra of dimension $n\ge 3$. \\[0.2cm]
Assume that $\Lr \cong V(2)\oplus \Ln_3(\C)$. We will look again at the possibilities for $\ker(j_1)$. If it has a non-trivial
Levi factor, then $\ker(j_1)=\Lg$, so that $j_2\colon \Lg\ra \Ln$ is bijective. This is a contradiction. Hence we may assume 
that $\ker(j_1)$ is solvable, so that it is contained in $V(2)\oplus \Ln_3(\C)$. Here we can view $\ker(j_1)$ both as a
subalgebra and as a submodule of $V(2)\oplus \Ln_3(\C)$. As an $\mathfrak{sl}_2(\C)$-submodule, $ V(2)\oplus \Ln_3(\C)$ is
isomorphic to $V(2)\oplus V(2)\oplus V(1)$, because $\Ln_3(\C)\cong V(2)\oplus V(1)$, see Example $\ref{2.2}$. The submodule
$V(2)\oplus V(2)$ cannot occur as an ideal, since there is no subalgebra corresponding to it. So we have the following
possibilities for $\ker(j_1)$ as a submodule: \\[0.2cm]
{\em Case 1:} $\ker(j_1)=0$. Then $j_1$ is an isomorphism. This is a contradiction. \\[0.2cm]
{\em Case 2:} $\ker(j_1)\cong V(2)$. Then $j_1\mid_{\mathfrak{sl}_2(\C)\ltimes \Ln_3(\C)}$ is injective. This is impossible,
because $\mathfrak{sl}_2(\C)\ltimes \Ln_3(\C)$ has a $3$-dimensional abelian subalgebra, see Example $\ref{2.2}$, but
$\mathfrak{sl}_3(\C)$ does not have one. \\[0.2cm]
{\em Case 3:} $\ker(j_1)\cong \Ln_3(\C)$. Since $\ker(j_1)\cap \ker(j_2)=0$, it follows that $j_2$ induces
an injective homomorphism
$\mathfrak{sl}_2(\C)\ltimes \ker(j_1) \hookrightarrow \mathfrak{sl}_3(\C)$, which is impossible as in case $2$. \\[0.2cm]
{\em Case 4:} $\ker(j_1)\cong V(2)\oplus V(1)$, with $V(1)\cong Z(\Ln_3(\C))$. Then $j_2$ is injective on $\ker(j_1)$,
which is impossible, because  $V(2)\oplus V(1)$ is an abelian subalgebra of dimension $3$. \\[0.2cm]
{\em Case 5:} $\ker(j_1)\cong V(2)\oplus \Ln_3(\C)$. Then $j_2\mid_{V(2)\oplus \Ln_3(\C)}$ is injective. This is impossible,
because $V(2)\oplus \Ln_3(\C)$ has a $3$-dimensional abelian subalgebra, but $\mathfrak{sl}_3(\C)$ does not have one. \\[0.2cm]
Assume that $\Lr \cong \Ln_5(\C)$. We claim that at least one of the maps $j_1,j_2$ must be injective on $\Ln_5(\C)$.
Otherwise $j_1(Z(\Ln_5(\C))=j_2(Z(\Ln_5(\C))=0$, so that $(j_1-j_2)\mid _{Z(\Ln_5(\C))}=0$, which is a contradiction to the fact that
$j_1-j_2$ is bijective. So we may assume that $j_1$ or $j_2$ is injective. This is impossible since $\Ln_5(\C)$ contains
a $3$-dimensional abelian subalgebra, but $\mathfrak{sl}_3(\C)$ does not have one. \\[0.2cm]
Finally assume that $\Lr \cong \Lf_{2,3}(\C)$. Then again at least one of the maps $j_1,j_2$ must be injective on $\Lr$.
If $j_1$ is not injective on $\Lr$, then $Z(\Lr)\cap \ker(j_1)\neq 0$. We claim that
$Z(\Lr)\subseteq \ker(j_1)$. In fact, every ideal $\La$ of $\Lg$ satisfying $\La\cap Z(\Lr) \neq 0$ also satisfies
$Z(\Lr)\subseteq \La$. To see this, note that
that $Z(\Lr)=[\Lr, [\Lr,\Lr]]$, and that the action of
$\mathfrak{sl}_2(\C)$ on $[\Lr,\Lr]/[\Lr, [\Lr,\Lr]]$ is trivial, since the quotient is $1$-dimensional. It follows
that the action of $\mathfrak{sl}_2(\C)$ on $Z(\Lr)$ coincides with the action on $\Lr/[\Lr,\Lr]$. By Proposition $\ref{2.1}$ it has
no trivial $1$-dimensional submodule, since $\Lg$ is perfect. This also implies that $Z(\Lr)$ has no trivial
$1$-dimensional submodule, so that $\dim (\La\cap Z(\Lr))\ge 2$. Since $\dim Z(\Lr)=2$ it follows that $Z(\Lr)\subseteq \La$. 
Hence if both $j_1$ and $j_2$ are not injective on $\Lr$, the center $Z(\Lr)$ is contained in $\ker(j_1)$ and $\ker(j_2)$, so that
$(j_1-j_2)\mid_{Z(\Lr)}=0$. This is a contradiction. Consequently, $j_i\colon \Lf_{2,3}(\C) \ra \mathfrak{sl}_3(\C)$
is an injection for some $i$, contradicting the fact that $\Lf_{2,3}(\C)$ has a $3$-dimensional abelian subalgebra, but
$\mathfrak{sl}_3(\C)$ does not have one.
\end{proof}  

It is not clear how to generalize this proof for other simple Lie algebras $\Ln$. \\[0.2cm]
For case $(e)$ we can prove the following general result.

\begin{prop}\label{3.10}
Let $(\Lg,\Ln)$ be a pair of Lie algebras, where $\Lg$ is perfect non-semisimple and $\Ln$ is semisimple. Assume that
we have a Levi decomposition $\Lg=\Ls \ltimes V$, where $\Ls$ is a simple subalgebra and $V$ is an irreducible $\Ls$-module,
considered as abelian Lie algebra. Then there is no PA-structure on $(\Lg,\Ln)$.
\end{prop}  

\begin{proof}
Suppose that there exists a PA-structure on $(\Lg,\Ln)$. Then by Proposition $\ref{2.7}$ it is of the form $x\cdot y=\{R(x),y\}$
for a Rota-Baxter operator $R\colon \Lg\ra \Ln$ of weight $1$ on $\Ln$. Moreover, since $\Lg$ and $\Ln$ are not isomorphic,
both $\ker(R)$ and $\ker(R+\id)$ are nonzero ideals of $\Lg$ with
\[
\ker(R)\cap \ker(R+\id)=0.
\]  
We will show that the only ideals of $\Lg=\Ls \ltimes V$ are $0,V,\Lg$. Then it is clear that $V$ is contained in the
above intersection, so that $V=0$. This is a contradiction. So let $\La$ be an ideal of $\Lg$. Then we obtain a Levi decomposition
for $\La$ by
\[
\La=(\La\cap \Ls)\ltimes (\La\cap V).
\]
Since $\La\cap \Ls$ is an ideal in $\Ls$, and $\Ls$ is simple, we have either $\La\cap \Ls=0$ or $\La\cap \Ls=\Ls$.
Also, since $\La\cap V$ is an $\Ls$-submodule of $V$ and $V$ is irreducible, we have either $\La\cap V=0$ or $\La\cap V=V$.\\[0.2cm]
{\em Case 1:} $\La\cap V=0$. Then $\La=\La\cap \Ls$ is either zero or $\Ls$. It follows that either $\La=0$ or $\La=\Ls$.
But $\Ls$ is not an ideal in $\Lg$, so that we obtain $\La=0$. \\[0.2cm]
{\em Case 2:} $\La\cap V=V$. Then $\La=(\La\cap \Ls)\ltimes V$, which means either $\La=V$ or $\La=\Lg$. \\[0.2cm]
So we obtain $V=0$ and hence a contradiction.
\end{proof}  

For $\Ln =\mathfrak{sl}_2(\C)\oplus \mathfrak{sl}_2(\C)$ we obtain a further result for case $(e)$.
In Theorem $4.1$ of \cite{BU59} we have classified all Lie algebras $\Lg$, such that the pair $(\Lg,\Ln)$ with
$\Ln =\mathfrak{sl}_2(\C)\oplus \mathfrak{sl}_2(\C)$ admits a PA-structure. Here we have used the theory of Rota-Baxter
operators. It is easy to see that none of the eight cases for $\Lg$ in this classification yields a perfect, non-semisimple
Lie algebra. Hence we obtain the following result.

\begin{prop}
Let $(\Lg,\Ln)$ be a pair of Lie algebras, where $\Lg$ is perfect non-semisimple and $\Ln =\mathfrak{sl}_2(\C)\oplus \mathfrak{sl}_2(\C)$. Then there is no PA-structure on $(\Lg,\Ln)$.
\end{prop} 

For case $(f)$ consider the perfect non-semisimple Lie algebra $\Lg=\mathfrak{sl}_2(\C)\ltimes V(2)$.
Let $(e_1,\ldots ,e_5)$ be a basis of $\Lg$ with $\mathfrak{sl}_2(\C)={\rm span}\{ e_1,e_2,e_3\}$, $V(2)={\rm span}\{ e_4,e_5\}$, and
Lie brackets given as follows:
\vspace*{0.5cm}
\begin{align*}
[e_1,e_2] & = e_3,     & [e_2,e_3] & = 2e_2,    & [e_3,e_4]& =e_4,\\
[e_1,e_3] & = -2e_1,   & [e_2,e_4] & = e_5,     & [e_3,e_5]& = -e_5, \\
[e_1,e_5] & = e_4.     &           &            &          &         \\
\end{align*}

\begin{ex}\label{3.13}
The pair of Lie algebras $(\Lg,\Ln)=(\mathfrak{sl}_2(\C)\ltimes V(2), \mathfrak{sl}_2(\C)\oplus \C^2)$ admits a PA-structure
given by
\begin{align*}
e_1\cdot e_5 & = e_4, \quad e_2\cdot e_4 = e_5,    \\
e_3\cdot e_4 & = e_4, \quad e_3\cdot e_5 = -e_5.    \\
\end{align*}
\end{ex}  
Here $\Ln$ is reductive with a $2$-dimensional center. Such examples are impossible when $\Ln$ is reductive
with a $1$-dimensional center, as the following result shows.

\begin{prop}
Let $(\Lg,\Ln)$ be a pair of Lie algebras, where $\Lg$ is perfect and $\Ln$ is reductive with a $1$-dimensional center.
Then there is no PA-structure on $(\Lg,\Ln)$.
\end{prop}

\begin{proof}
Assume that there exists a PA-structure $x\cdot y=L(x)(y)$ on $(\Lg,\Ln)$. Then by Proposition 
$2.11$ in \cite{BU41} we have an injective Lie algebra homomorphism 
\[
\phi\colon \Lg \hookrightarrow \Ln \rtimes \Der (\Ln),\; x\mapsto (x,L(x)).
\]
Writing $\Lh=L(\Lg)$ we obtain a direct vector space decomposition $\Ln\rtimes \Lh=\phi(\Lg)\dotplus \Lh$. Note that $\Lh$
is nonzero. Since $\Lg$ is perfect and $\phi$ and $L$ are homomorphisms, $\phi(\Lg)$ and $\Lh$ are perfect subalgebras of
$\Ln\rtimes \Lh$. Hence also $\phi(\Lg)\dotplus \Lh$ is perfect, see the proof of Lemma $2.3$ in \cite{BU74}, so that
$\Ln\rtimes \Lh$ is perfect. By assumption we have $\Ln=[\Ln,\Ln]\oplus Z(\Ln)$, where $\Ls=[\Ln,\Ln]$ 
is semisimple and $Z(\Ln)$ is $1$-dimensional. Since the commutator and the center of $\Ln$ are 
characteristic ideals in $\Ln$, and $\Ln$ is an ideal in $\Ln\rtimes \Lh$, both $[\Ln,\Ln]$ 
and $Z(\Ln)$ are ideals in $\Ln\rtimes \Lh$. We claim that $[Z(\Ln),\Lh]=0$ for the
Lie bracket in $\Ln\rtimes \Lh$. Since $Z(\Ln)$ is an ideal in  $\Ln\rtimes \Lh$, we have $[\Lh,Z(\Ln)]\subseteq Z(\Ln)$,
so that $Z(\Ln)$ is a $1$-dimensional $\Lh$-module. However, for a perfect Lie algebra, every 
$1$-dimensional module is trivial. The proof is the same as for a semisimple Lie algebra. 
Hence we obtain $[Z(\Ln),\Lh]=0$. It follows that
\[
[\Ln,\Lh]=[\Ls+Z(\Ln),\Lh]=[\Ls,\Lh]\subseteq \Ls,
\]
since $\Ls$ is an ideal in $\Ln\rtimes \Lh$. 
Because $\Ln\rtimes \Lh$ is perfect, we have
\begin{align*}
\Ln+\Lh & =[\Ln+\Lh,\Ln+\Lh] \\
        & =[\Ln,\Ln]+[\Ln,\Lh]+[\Lh,\Lh] \\
        & = \Ls+\Lh.
\end{align*}
However, because of $\Ln=\Ls\oplus Z(\Ln)$ we have $\dim(\Ln)=\dim(\Ls)+1$. Since
$\Ln\cap \Lh=\Ls\cap \Lh=0$, this implies $\dim (\Ln+\Lh)=\dim(\Ls+\Lh)+1$. This is a contradiction
to $\Ln+\Lh=\Ls+\Lh$.
\end{proof}

For case $(g)$ we have the following result.

\begin{prop}
Let $(\Lg,\Ln)$ be a pair of Lie algebras, where $\Lg$ is perfect and $\Ln$ is complete non-perfect.
Then there is no PA-structure on $(\Lg,\Ln)$.
\end{prop}

\begin{proof}
Assume that there exists a PA-structure on $(\Lg,\Ln)$. Since $\Ln$ is complete, this
PA-structure is given by $x\cdot y=\{R(x),y\}$ for a Rota-Baxter operator $R$ of weight $1$.
Because $\Lg$ is perfect, it follows by Corollary $2.20$ in \cite{BU59} that $\Ln$ is also 
perfect. This is a contradiction. 
\end{proof}

\section{PA-structures with $\Ln$ perfect}

In this section we study the existence question of PA-structures on pairs of complex Lie algebras $(\Lg,\Ln)$, where
$\Ln$ is perfect non-semisimple. We consider $7$ different cases for $\Lg$, namely $(a)$ $\Lg$ is abelian,
$(b)$ $\Lg$ is nilpotent non-abelian, $(c)$ $\Lg$ is solvable non-nilpotent, $(d)$ $\Lg$ is simple,
$(e)$ $\Lg$ is semisimple non-simple, $(f)$ $\Lg$ is reductive non-semisimple, and $(g)$ $\Lg$ is complete non-perfect.\\[0.2cm]
For case $(a)$ we have the following result.

\begin{prop}
There is no PA-structure on a pair $(\Lg,\Ln)$, where $\Lg$ is abelian and $\Ln$ is perfect.
\end{prop}

\begin{proof}
Any PA-structure on $(\Lg,\Ln)$ with $\Ln$ abelian corresponds to an LR-structure on $\Lg$. However, every
Lie algebra admitting an LR-structure is $2$-step solvable by Proposition $2.1$ in \cite{BU34}. 
Hence there exists no PA-structure on $(\Lg,\Ln)$. 
\end{proof}  

For case $(b)$ we have the following result.

\begin{prop}
There is no PA-structure on a pair $(\Lg,\Ln)$, where $\Lg$ is nilpotent non-abelian and $\Ln$ is perfect.
\end{prop}

\begin{proof}
Assume that there is a PA-structure on $(\Lg,\Ln)$. Since $\Lg$ is nilpotent, $\Ln$ must be solvable by
Proposition $4.3$ in \cite{BU41}. This is a contradiction.
\end{proof}

For case $(c)$, consider the perfect non-semisimple Lie algebra $\Ln=\mathfrak{sl}_2(\C)\ltimes V(2)$ with the Lie brackets
given before Example $\ref{3.13}$. We have a decomposition $\Ln=\Ln_1\dotplus \Ln_2$ into subalgebras
$\Ln_1={\rm span} \{e_1,e_4,e_5\}$ and $\Ln_2={\rm span} \{e_2,e_3\}$. Then by Propositions $2.7$ and $2.13$ in \cite{BU59},
the Rota-Baxter operator $R$ given by $R(x+y)=-y$ for all $x\in \Ln_1, y\in \Ln_2$ defines a PA-structure on the pair $(\Lg,\Ln)$, where
$\Lg\cong \Ln_1\oplus \Ln_2\cong \Ln_3(\C)\oplus \Lr_2(\C)$ is solvable non-nilpotent. The matrix of $R$ is given by
\[
R=\begin{pmatrix} 0 & 0 & 0 & 0 & 0 \cr
                  0 & -1 & 0 & 0 & 0 \cr
                  0 & 0 & -1 & 0 & 0 \cr
                  0 & 0 & 0 & 0 & 0 \cr
                  0 & 0 & 0 & 0 & 0 \end{pmatrix},
\]
and the Lie brackets of $\Lg$ are given by $[e_1,e_5]=e_4$ and $[e_2,e_3]=-2e_2$.        

\begin{ex}\label{4.3}
The pair of Lie algebras $(\Lg,\Ln)$ with $\Lg\cong \Ln_3(\C)\oplus \Lr_2(\C)$ and $\Ln= \mathfrak{sl}_2(\C)\ltimes V(2)$
admits a PA-structure given by
\begin{align*}
e_2\cdot e_1 & = e_3,     & e_3\cdot e_1 & = -2e_1,    & e_3 \cdot e_4 & =-e_4,\\
e_2\cdot e_3 & = -2e_2,   & e_3\cdot e_2 & = 2e_2,     & e_3\cdot e_5 & = e_5, \\
e_2\cdot e_4 & = -e_5.    &         &           &          &         
\end{align*}
\end{ex}  

For the cases $(d)$ and $(e)$ we have the following result.

\begin{prop}
There is no PA-structure on a pair $(\Lg,\Ln)$, where $\Lg$ is semisimple and  $\Ln$ is perfect non-semisimple.
\end{prop}

\begin{proof}
Assume that there exists a PA-structure on  $(\Lg,\Ln)$, where $\Lg$ is semisimple. Then by Theorem $3.3$ in \cite{BU73},
$\Lg$ is isomorphic to $\Ln$. This is a contradiction.  
\end{proof}  

For case $(f)$ we have the following example.

\begin{ex}\label{4.5}
The pair of Lie algebras $(\Lg,\Ln)=( \mathfrak{sl}_2(\C)\oplus \C^2, \mathfrak{sl}_2(\C)\ltimes V(2))$ admits a PA-structure
given by
\begin{align*}
e_4\cdot e_2 & = e_5, \quad e_5\cdot e_1 = e_4,    \\
e_4\cdot e_3 & = e_4, \quad e_5\cdot e_3 = -e_5.    \\
\end{align*}
\end{ex}  
Here we use the Lie brackets for  $\mathfrak{sl}_2(\C)\ltimes V(2)$ as in Example $\ref{3.13}$, and the standard Lie brackets
of $\mathfrak{sl}_2(\C)={\rm span}\{ e_1,e_2,e_3\}$ for $\Lg$. This PA-structure can also be realized by the Rota-Baxter operator
\[
R=\begin{pmatrix} 0 & 0 & 0 & 0 & 0 \cr
                  0 & 0 & 0 & 0 & 0 \cr
                  0 & 0 & 0 & 0 & 0 \cr
                  0 & 0 & 0 & -1 & 0 \cr
                  0 & 0 & 0 & 0 & -1 \end{pmatrix}
\]
for the decomposition $\Ln=\Ln_1\dotplus \Ln_2$, where $\Ln_1={\rm span} \{ e_1,e_2,e_3 \}$, 
$\Ln_2={\rm span} \{ e_4,e_5 \}$, $\Ln=\mathfrak{sl}_2(\C)\ltimes V(2)$ and $x\cdot y=\{R(x),y\}$. \\[0.2cm]
Finally, for case $(g)$ we have the following example.

\begin{ex}\label{4.6}
The pair of Lie algebras $(\Lg,\Ln)=(\mathfrak{sl}_2(\C)\oplus \Lr_2(\C), \mathfrak{sl}_2(\C)\ltimes V(2))$
admits a PA-structure given by
\begin{align*}
e_4\cdot e_2 & = e_5,   & e_5\cdot e_1 & = e_4,    & e_5 \cdot e_4 & =-e_4,\\
e_4\cdot e_3 & = e_4,   & e_5\cdot e_3 & = -e_5,   & e_5\cdot e_5 & = -e_5, \\
\end{align*}
\end{ex}
Here $\Lg$ is complete non-perfect, and the Lie brackets of $\Lg$ are given by the standard brackets for
$\mathfrak{sl}_2(\C)={\rm span} \{ e_1,e_2,e_3 \}$, and by $[e_4,e_5]=e_4$ for $\Lr_2(\C)={\rm span}\{ e_4,e_5 \}$.

\section{The existence question}

We summarize the existence results for post-Lie algebra structures
from the previous sections and from the papers \cite{BU41,BU44,BU51,BU58,BU59,BU64,BU73} as follows. 

\begin{thm}
The existence table for post-Lie algebra structures on pairs $(\Lg,\Ln)$ is given as follows:
\vspace*{0.5cm}
\begin{center}
\begin{tabular}{l|llllllll}
$(\Lg,\Ln)$ & \color{blue}{$\Ln$ abe} & \color{ForestGreen}{$\Ln$ nil} & $\Ln$ sol & \color{red}{$\Ln$ sim}
& \color{Dandelion}{$\Ln$ sem} & \color{purple}{$\Ln$ red} &  \color{cyan}{$\Ln$ com} & \color{Fuchsia}{$\Ln$ per} \\[1pt]
\hline
\color{blue}{$\Lg$ abelian} & $\ck$ & $\ck$ & $\ck$ & $-$ & $-$ & $-$ & $\ck$ & $-$ \\[1pt]
\color{ForestGreen}{$\Lg$ nilpotent} & $\ck$ & $\ck$ & $\ck$ & $-$ & $-$ & $-$ & $\ck$ & $-$ \\[1pt]
 $\Lg$ solvable & $\ck$ & $\ck$ & $\ck$ & $\ck$ & $\ck$ & $\ck$ & $\ck$ & $\ck$ \\[1pt]
\color{red}{$\Lg$ simple} & $-$ & $-$ & $-$ & $\ck$ & $-$ & $-$ & $-$ &  $-$ \\[1pt]
\color{Dandelion}{$\Lg$ semisimple} & $-$ & $-$ & $-$ & $-$ & $\ck$ & $-$ & $-$ &  $-$ \\[1pt]
\color{purple}{$\Lg$ reductive}  & $\ck$ & $\ck$ & $\ck$ & $-$ & \color{red}{$?$} & $\ck$ & $\ck$ & $\ck$ \\[1pt]
\color{cyan}{$\Lg$ complete}  & $\ck$ & $\ck$ & $\ck$ & $\ck$ & $\ck$ & $\ck$ & $\ck$ & $\ck$ \\[1pt]
\color{Fuchsia}{$\Lg$ perfect}  & $-$ & \color{red}{$?$} & $-$ & \color{red}{$?$} & \color{red}{$?$} & $\ck$ & $-$ &  $\ck$ \\[1pt]  
\end{tabular}
\end{center}
\end{thm}

A checkmark only means that there is {\em some} non-trivial pair $(\Lg,\Ln)$ of Lie algebras with the
given algebraic properties admitting a PA-structure. A dash means that there does not exist 
any PA-structure on such a pair. Recall that the classes are (to avoid unnecessary overlap) 
abelian, nilpotent non-abelian, solvable non-nilpotent, simple, semisimple non-simple, reductive 
non-semisimple and complete non-perfect.

\section*{Acknowledgments}
Dietrich Burde and Mina Monadjem are supported by the Austrian Science Foun\-da\-tion FWF, grant P 33811. 
Karel Dekimpe is supported by Methusalem grant Meth/21/03 - long term structural funding of the Flemish Government.

\end{document}